\documentstyle[11pt,amssymb,amsmath,amscd,amsbsy,amsfonts,theorem,hyperref]{article}

\input xy
\xyoption{all}

\newcommand{\zz}{{\Bbb Z}}

\newcommand{\qq}{{\Bbb Q}}
\newcommand{\pp}{{\Bbb P}}

\newcommand{\iis}{{\mathbf{i}}}

\newcommand{\ddim}{\operatorname{dim}}

\newcommand{\op}[1]{\operatorname{#1}}
\newcommand{\kbar}{\overline{k}}

\newcommand{\ffi}{\varphi}

\newcommand{\row}{\rightarrow}

\newcommand{\lrow}{\longrightarrow}
\renewcommand{\leq}{\leqslant}
\renewcommand{\geq}{\geqslant}

\newcommand{\nichego}[1]{}

\newcommand{\ov}[1]{\overline{#1}}

\newcommand{\laz}{{\Bbb L}}
\newcommand{\llaz}{{\mathbf{L}}}

\newcommand{\cl}{{\cal L}}
\newcommand{\cm}{{\cal M}}

\newcommand{\Qed}{\hfill$\square$\smallskip}
\newenvironment{proof}{\noindent{\it Proof}:}{\vskip 5mm}

\newtheorem{prop}{Proposition}[section]{\bf}{\it}
\newtheorem{thm}[prop]{Theorem}{\bf}{\it}
{\bf}{\it}
{\bf}{\it}
{\bf}{\it}
\newtheorem{conj}[prop]{Conjecture}{\bf}{\it}
{\bf}{\it}
\newtheorem{exa}[prop]{Example}{\bf}{\it}
\newtheorem{rem}[prop]{Remark}{\bf}{}

{\bf}{\it}
{\bf}{\it}
\newtheorem{cor}[prop]{Corollary}{\bf}{\it}
{\bf}{\it}

\begin{document}

\title{Algebraic Cobordism as a module over the Lazard ring}
\author{Alexander Vishik\footnote{School of Mathematical Sciences, University
of Nottingham}}
\date{}

\maketitle

\begin{abstract}
In this paper we study the structure of the Algebraic Cobordism ring of a variety
as a module over the Lazard ring, and show that it has relations in positive codimensions.
We actually prove the stronger graded version.
This extends the result of M.Levine-F.Morel \cite{LM} claiming that this module has generators
in non-negative codimensions. As an application we compute the Algebraic Cobordism ring of a curve.
The main tool is Symmetric Operations in Algebraic Cobordism - \cite{SOpSt}.
\end{abstract}

\tableofcontents

\section{Introduction}
\label{Intro}
Cohomology theories are used to provide invariants of algebraic varieties.
The {\it Algebraic Cobordism} theory of Levine-Morel $\Omega^*$ - \cite{LM}
is the universal oriented cohomology
theory in the algebro-geometric context. It is much larger than
$\op{CH}^*$ and $K_0$, and contains these classical theories as little faces.
The fact that rationally any formal group law, including the universal one, is isomorphic to
an additive FGL implies that rationally $\Omega^*_{\qq}$ splits into the direct sum of
copies of $\op{CH}^*_{\qq}$, and for the respective graded (by codimension of support)
theory we have a natural identification:
$$
Gr\Omega^*_{\qq}=\op{CH}^*_{\qq}\otimes_{\zz}\laz.
$$
So, all the relations in $Gr\Omega^*$ (modulo the relations in Chow groups)
are torsion. But integrally, the structure of Algebraic Cobordism is much more
complex, and little is known about it. The number of types of non-cellular
varieties for which $\Omega^*(X)$ was computed is close to zero.
In this article we study the general structure of $\Omega^*(X)$ as a module over the
{\it Lazard ring}. Namely, the relations of this module.
It appears that these relations are generated by those
concentrated in positive codimensions - Theorem \ref{osn}.
This fits nicely with the result of M.Levine-F.Morel - \cite{LM} claiming
that the generators of this module are in non-negative codimensions.
Thus, we have an exact sequence of graded $\laz$-modules:
$$
M_1\row M_0\row\Omega^*(X)\row 0,
$$
where $M_0$ and $M_1$ are free modules with generators concentrated in nonnegative/positive codimensions,
respectively.

As an application we compute the Algebraic Cobordism ring of a curve - Theorem \ref{curve}.
We also make a guess how the {\it syzygies} of our $\laz$-module should behave - see Conjecture \ref{syz}
and the discussion after it.

Our main result follows from the stronger graded version. The main tool here is the use of the
{\it Symmetric operations} - see \cite{SOpSt}. These unstable operations are more subtle than the
{\it Landweber-Novikov operations} and permit to work effectively with the torsion effects,
which the current article should demonstrate.

Everywhere below we assume that the characteristic of the base field $k$ is zero.

{\bf Acknowledgements:} I would like to thank the Referee for useful suggestions which substantially
improved the exposition.

\section{$BP$-theory}
\label{BP}

The Algebraic Cobordism theory $\Omega^*$ of Levine-Morel is universal in the sense that it has unique
morphism $\Omega^*\stackrel{\rho_A}{\lrow}A^*$ to any other oriented cohomology theory - \cite[Th.1.2.6]{LM}
(for general facts on oriented theories see \cite{LM,PS,SU}).
For a smooth variety $X$, the ring $\Omega^*(X)$ is additively generated by the classes
of projective maps $[V\stackrel{v}{\row}X]$ from smooth varieties, subject to certain relations - see
\cite{LM,LP}. It has natural dimensional = $\ddim(V)$, and codimensional = $(\ddim(X)-\ddim(V))$ gradings
(preserved by pushforwards, respectively pullbacks).

To each oriented theory $A^*$ one can assign the {\it Formal Group Law} $(A,F_A)$, where
$A=A^*(\op{Spec}(k))$ is the {\it coefficient ring of the theory},
and $F_A\in A[[x,y]]$ describes how to compute the Chern class of the tenzor product
of line bundles: $c^A_1(\cl\otimes\cm)=F_A(c^A_1(\cl),c^A_1(\cm))$ - it can be defined using the
{\it Segre embedding} - see, for example, \cite[1.1.3]{LM}, or \cite[2.3]{SU} (the Chern classes
come with the orientation - see \cite{PS} and \cite{LM}).
In the case of Algebraic Cobordism, the respective FGL will be the {\it universal} one
$(\laz,F_U)$. In particular, $\Omega^*(\op{Spec}(k))$ is naturally the Lazard ring
$\laz\cong\zz[x_1,x_2,\ldots]$ (one polynomial generator in every positive dimension) - see
\cite[Th.1.2.7]{LM}. This ring is additively generated by the classes of smooth projective varieties,
with some relations. It also coincides with the $MU^*(pt)$ from Topology.

There is canonical {\it Hurewicz} embedding $\laz\hookrightarrow\zz[b_1,b_2,\ldots]$ corresponding
(via universality of $(\laz,F_U)$) to the formal group law over $\zz[b_1,b_2,\ldots]$
which is obtained from the additive
one by the generic ("stable") change of variable $y=x+b_1x^2+b_2x^3+\ldots$.
Rationally, it is an isomorphism. And the coefficients at particular monomials $\ov{b}^{\ov{I}}$ give
us linear functions $f_{\ov{I}}:\laz\row\zz$. These are the {\it characteristic numbers} - computed on the
class of the variety $X$ they give us the characteristic numbers of $X$ (alternatively, can be obtained
as degrees of zero-cycles given by various polynomials in Chern classes of minus tangent bundle $(-T_X)$).

Algebraic Cobordism theory combines effects related to all prime numbers.
It is convenient to separate the information concerning particular prime $p$
in order to simplify the situation.
In Topology this is achieved with the help of the {\it Brown-Peterson theory} $BP^*$
constructed in \cite{BP}. It is obtained from the localized at $p$
{\it complex-oriented cobordism} $MU_{\zz_{(p)}}^*$ by an explicit multiplicative projector
produced out of {\it Adams operations}
(see, for example, \cite{Ad} and \cite[I.3]{Wi}).
Exactly the same projector $\rho$ 
(characterized by the property that $\rho([\pp^n])=[\pp^n]$, if $n+1$ is a power of $p$,
and zero otherwise)
can be applied in the algebro-geometric context to $\Omega^*_{\zz_{(p)}}$
producing the algebraic variant of $BP^*$ which we still call by the same name.
Then, as in Topology - see \cite[Th.4]{Q}, $\Omega^*_{\zz_{(p)}}$ additively splits
into a direct sum of copies of $BP^*$.
It follows from \cite[Pr.4.11]{SU} that $BP^*$ can be obtained from Algebraic Cobordism
$\Omega^*$ of Levine-Morel by change of coefficients:
$BP^*(X)=\Omega^*(X)\otimes_{\laz}BP$.
The coefficient ring $BP=BP^*(\op{Spec}(k))$ will be the same as for the topological counterpart.

Namely, $BP=\zz_{(p)}[P_1,P_2,\ldots]$, where $P_i$ is a $\nu_i$-element. That is,
$\ddim(P_i)=p^i-1$,
all characteristic numbers of $P_i$ are divisible by $p$, and the only "additive"
characteristic number (corresponding to $b_1^{p^i-1}$) of it is not divisible by $p^2$
(where we consider $BP$ as a subring of $\laz$).
We can also set $P_0=p$.
The choice of the above generators is not unique, but they
can be selected as the coefficients of the formal $p$ (see \cite[Lem.3.17]{Wi}):
\begin{equation}
\label{[p]}
[p]=\frac{p\cdot_{BP}t}{t}\equiv\sum_{l\geq 0}P_lt^{p^l-1}\,\,(\,mod\,I(p)^2),
\end{equation}
where $p\cdot_{BP}t$ is the multiplication by $p$ in the sense of the FGL $F_{BP}$, and
$I(p)$ is the ideal consisting of elements whose all characteristic numbers are divisible by $p$,
that is, the ideal generated by $P_i,\,i\geq 0$.
It is convenient to introduce the notation:
$$
[p]_{\leq i}:=\sum_{l=0}^iP_lt^{p^l-1}\in BP[t].
$$

\section{Symmetric and Steenrod Operations}
\label{SO}

Our principal method is based on the use of {\it Symmetric operations} in algebraic cobordism.
These operations are related to {\it Steenrod operations} of Quillen's type there.
Let me start with the latter (where we need a $\zz_{(p)}$-local version only).

Recall that an {\it operation} between theories is a transformation commuting with the pull-back maps.
Among them, the most structured are {\it multiplicative operations}, that is, operations acting as
homomorphisms of rings.
To each multiplicative operation $A^*\stackrel{G}{\row}B^*$ one can assign the morphism of the
respective FGLs:
$$
\begin{CD}
(A,F_A)& @>{(\ffi_G,\gamma_G)}>>&(B,F_B),
\end{CD}
$$
where $\ffi_G:A\row B$ is the restriction of $G$ to the $\op{Spec}(k)$, and
$\gamma_G(x)\in B[[x]]\cdot x$ describes the action of $G$ on the Chern classes:
$G(c^A_1(\cl))=\gamma_G(c^B_1(\cl))$.
The fact that $(\ffi_G,\gamma_G)$ is a morphism of FGLs amounts to the equation:
\begin{equation}
\label{mFGL}
\ffi_G(F_A)(\gamma_G(u),\gamma_G(v))=\gamma_G(F_B(u,v)).
\end{equation}
In a good situation, such as ours, there is a $1$-to-$1$ correspondence between
morphisms of FGLs and multiplicative operations - see \cite[Th.6.8]{SU}.
And in the case where the source theory is $\Omega^*$, everything is determined by $\gamma_G$ alone
(since $(\laz,F_{\Omega})$ is the {\it universal FGL}),
moreover, one can choose $\gamma_G$ in an arbitrary way, as long as the leading coefficient $b_0$ of it
is invertible (here you even do not need \cite{SU}, but can use the {\it universality} of $\Omega^*$ due
to M.Levine-F.Morel - \cite[Th.1.2.6]{LM} and the {\it reparametrization} of I.Panin-A.Smirnov -
\cite{PS,P,Sm1}).

Let $p$ be a prime. Choose representatives $\ov{i}=\{i_1,\ldots,i_{p-1}\}$ of all non-zero cosets
$(\,mod\,p)$, and denote as $\iis$ their product.
Acting as above, we obtain the multiplicative {\it Total Steenrod operation}:
\begin{equation}
\label{St}
\Omega^*_{\zz_{(p)}}\stackrel{St(\ov{i})}{\lrow}\Omega^*_{\zz_{(p)}}[[t]][t^{-1}],
\end{equation}
with the {\it inverse Todd genus}
$\frac{\gamma_{St(\ov{i})}(x)}{x}=\prod_{l=1}^{p-1}(x+_{\Omega}(i_l)\cdot_{\Omega}t)$ - see \cite[6.4]{SOpSt} (note, that our leading coefficient is invertible).

Any multiplicative operation $G$ (from $\Omega^*$ elsewhere) whose Todd genus starts with $1$
({\it stable} operation, in other words) is a specialization of the {\it Total Landweber-Novikov operation} -
the operation corresponding to the "generic stable" change of variable
$\gamma_{S^{Tot}_{L-N}}=x+b_1x^2+b_2x^3+\ldots$ - \cite[Ex.4.1.25]{LM}:
$$
\xymatrix @-0.2pc{
\Omega^* \ar @{->}[rr]^(0.4){S_{L-N}^{Tot}} \ar @{->}[rrd]_{G} & &
\Omega^*[b_1,b_2,\ldots] \ar @{->}[d] \\
& & B^*.
}
$$
Here one uses the universality of $\Omega^*$,
and maps the independent variables $b_1,b_2,\ldots$ to the coefficients of $\gamma_G$.
In other words, we apply $\rho_B$, and plug particular values for $b_i$'s into the formula:
$$
S^{Tot}_{L-N}=\sum_{\ov{I}}\ov{b}^{\ov{I}}S^{\ov{I}}_{L-N}.
$$
If the Todd genus starts with the unit $b_0$, then on $\Omega^n$ our operation is a composition of the
specialization as above (corresponding to $b_0^{-1}\gamma_G$) with the multiplication by $b_0^n$.
I would call this the {\it generalized specialization}.
In the case of $St(\ov{i})$, the leading coefficient of our Todd genus is a power series in $t$ starting
with $\iis\cdot t^{p-1}$, and so is invertible (in the target theory).
Consequently, $St(\ov{i})$ is a generalized specialization of the Total Landweber-Novikov operation,
and in particular,
all components of $St(\ov{i})$ are $\laz_{(p)}$-linear combinations (infinite, in general) of the Landweber-Novikov operations (the fact that the combinations are infinite should not scare the reader, since,
for each particular variety, the infinite tail of the expression will act trivially by dimensional
considerations). Notice, that these expressions will depend on the degree $n$ of the element, which reflects
the fact that the operation is unstable.

Below the choice of the coset representatives will not be important, so I will omit it from
the notations.

Denote as $\square^p$ the $p$-th power operation. Then it can be shown - see \cite{SOpSt} that
the non-positive degree (in $t$) part of the operation $(\square^p-St)$ is divisible by formal $[p]$.
Moreover, it was shown in \cite[Th.7.1]{SOpSt} that
one can divide canonically and obtain the {\it Total Symmetric operation} for a given prime:
$$
\Omega_{\zz_{(p)}}\stackrel{\Phi}{\lrow}\Omega_{\zz_{(p)}}[t^{-1}].
$$
Such an operation controls all $p$-primary divisibilities of characteristic numbers,
and this is exactly what we are exploiting.
Operation $\Phi$ is not multiplicative, but is almost additive - it becomes additive if
we cut off the component of degree zero.

The above operations can be extended from $\Omega^*$ to any theory obtained from it by a
{\it multiplicative projector}. In particular, to the theory $BP^*$.
Considering the composition
$$
BP^*\hookrightarrow\Omega_{\zz_{(p)}}^*\stackrel{St}{\row}
\Omega_{\zz_{(p)}}^*[[t]][t^{-1}]\twoheadrightarrow
BP^*[[t]][t^{-1}]
$$
we get the multiplicative operation on $BP^*$-theory which we still denote as $St$.
Similarly, we get a Total Symmetric Operation $BP^*\stackrel{\Phi}{\row}BP^*[t^{-1}]$.

\begin{prop}
\label{StP}
Let $\prod_{l=1}^mP_{k_l}$ be some monomial of dimension $d=\sum_{l=1}^m(p^{k_l}-1)$. Then
$$
St(\prod_{l=1}^mP_{k_l})\equiv t^{-pd}\cdot\prod_{l=1}^m[p]_{\leq k_l}\,\,\,(\,mod\,I(p)^{m+1}).
$$
\end{prop}

\begin{proof}
Since $St$ is multiplicative, it is sufficient to prove the case $m=1$.
Recall that $\gamma_{St}(x)=x\cdot\prod_{j=1}^{p-1}(x+_{BP}i_j\cdot_{BP}t)$.
Since we are working $(\,mod\,I(p)^2)$, the choice of the coset representatives $i_j$
will not be important for us.
Indeed, the different choices will give us multiplicative operations with $\gamma$'s
congruent modulo $I(p)$. Hence, (the coefficients of) the respective expressions in terms of the Landweber-Novikov
operations will be also congruent modulo $I(p)$. But $P_k\in I(p)$, and this ideal is
stable under (the $BP$-projections of) the Landweber-Novikov operations.

By the same reason, we can substitute the FGL in the formula for $\gamma_{St}$ by the {\it additive}
FGL $(\,mod\, p)$ (because all elements of positive dimensions in $BP$, as well as $p$, belong to $I(p)$,
and so the FGL of $BP$-theory is $I(p)$-congruent to the additive $(\,mod\,p)$ one).
In other words, we can assume that $\gamma(x)=x^p-xt^{p-1}$.
Since $St$ is multiplicative, and so defines a morphism of $FGL$'s,
from (\ref{mFGL}) we get the equation:
$$
p\cdot_{\ffi_{St}(BP)}\gamma(x)=\gamma(p\cdot_{BP}x),
$$
from which we obtain (using (\ref{[p]}) and the fact that $I(p)^2$ is stable under
Landweber-Novikov operations):
$$
\sum_{l\geq 0}St(P_l)(x^p-xt^{p-1})^{p^l}\equiv
\left(\sum_{l\geq 0}P_lx^{p^l}\right)^p-\left(\sum_{l\geq 0}P_lx^{p^l}\right)t^{p-1}
\hspace{5mm}(\,mod\,I(p)^2).
$$
Comparing coefficients at $x^{p^{k}}$ we get:
$$
St(P_{k-1})-St(P_k)t^{p^k(p-1)}\equiv-P_kt^{p-1}\,\,\,\,\,(\,mod\,I(p)^2).
$$
This immediately gives:
\begin{equation*}
\begin{split}
St(P_k)\equiv & P_kt^{(p^k-1)-p(p^k-1)}+P_{k-1}t^{(p^{k-1}-1)-p(p^k-1)}+\ldots+P_0t^{-p(p^k-1)}\\
\equiv & t^{-p(p^k-1)}\cdot [p]_{\leq k}\,\,\,\,\,(\,mod\,I(p)^2).
\end{split}
\end{equation*}
\Qed
\end{proof}

In particular, modulo $I(p)^{m+1}$, $St(I(p)^m)$ is concentrated in
non-positive degrees of $t$.

\begin{cor}
\label{StId}
The $t^{-d(p-1)}$-component of $St$ is the identity map on the dimension $d$ component of
$I(p)^m/I(p)^{m+1}$, for all $m$.
\end{cor}

\begin{prop}
\label{SymIm}
$$
\Phi(I(p)^{m+1})\subset I(p)^m[t^{-1}].
$$
\end{prop}

\begin{proof}
Let $x\in I(p)^{m+1}$.
Since $[p]\cdot\Phi(x)\equiv(x^p-St(x))\,\,(\,\text{modulo positive powers of}\,\,t)$,
from Proposition \ref{StP} we get that
$[p]\cdot\Phi(x)\in I(p)^{m+1}[t^{-1}]\,\,(\,\text{modulo positive powers of}\,\,t)$.
Since $P_i$ are algebraically independent variables in $BP$, we obtain by an increasing induction
on the power of $t$ that all coefficients of $\Phi(x)$ belong to $I(p)^m$.
\Qed
\end{proof}

\section{Graded Algebraic Cobordism}
\label{GAC}

On $\Omega^*$, as well as on any other theory, we have the filtration by codimension of support:
$u\in F^r\Omega^*(X)$, if it vanishes when restricted to an open set $U$, whose complement has
codimension $\geq r$. This filtration is clearly respected by all cohomological operations (mapping
zero to zero).
Thus, such operations also act on the respective graded theory $Gr\Omega^*$ (for a non-additive
operation need to consider a "discrete derivative" here - see \cite{PO}).
Moreover, this action is substantially simpler, which is exploited below.

The graded theory $Gr\Omega^*$ is a direct sum of it's graded components $\Omega_{(r)}^*$,
each of which has the structure of a graded $\laz$-module (as the $\laz$-module structure on $\Omega^*(X)$
respects our filtration). Thus, we have two different gradings
in the picture. We are using the one which is compatible with the non-graded ($\Omega^*$)
version and with the
$\laz$-module structure, and call it {\it codimensional} (as opposed to "support codimensional").

\begin{thm}
\label{osnGr}
Let $X$ be a smooth quasi-projective variety. Then $Gr\Omega^*(X)$ as a module over $\laz$
can be defined by generators of non-negative codimensions and relations of positive codimensions.
That is, there is a presentation
$$
N_1\row N_0\row Gr\Omega^*(X)\row 0
$$
of our graded cobordism as the cokernel of free graded $\laz$-modules, where the
grading on $Gr\Omega^*(X)$ is the codimensional one,
the generators of $N_0$ have non-negative degrees, while the generators of $N_1$ have positive degrees.
\end{thm}

\begin{proof}
Due to the result of M.Levine-F.Morel - \cite[Cor.4.5.8]{LM}, we have the surjection:
\begin{equation}
\label{surj}
\op{CH}(X)\otimes_{\zz}\laz\stackrel{\rho}{\twoheadrightarrow} Gr\Omega(X).
\end{equation}
Here we already imposed some relations of codimensions $r>0$, as $\op{CH}^r(X)$ is not free abelian, in
general.
I claim that the kernel of this map as an $\laz$-module is generated in positive codimensions.
It is sufficient to prove this statement $\otimes\zz_{(p)}$, for all primes $p$.
Then $\Omega_{\zz_{(p)}}^*$ and $Gr\Omega_{\zz_{(p)}}^*$ split into direct sum of copies
of $BP^*$ and $GrBP^*$. So, it is sufficient to prove the statement for $GrBP^*$.

The problem splits into ones corresponding to various graded pieces, where
the component $BP^*_{(r)}$ of "support codimension" $r$ of $GrBP$-theory is covered as
(the $BP$-projection of (\ref{surj})):
$$
\op{CH}^r(X)\otimes_{\zz}BP\stackrel{\rho}{\twoheadrightarrow} BP_{(r)}^*(X).
$$
$BP^*$ is a {\it constant} theory, that is, for any variety $X$,
$BP^*(\op{Spec}(k(X)))=BP^*(\op{Spec}(k))$ (see \cite[Def.4.4.1, Lem.4.4.2]{LM}, and
\cite[2.1, Pr.4.11]{SU}). This means that the pull back map $BP^*(\op{Spec}(k))\row BP^*(X)$
has a splitting - the restriction to the generic point $BP^*(X)\row BP^*(\op{Spec}(k(X)))$, the kernel of which consists of elements with positive codimension of support.  Hence,
$BP^*_{(0)}=BP\cdot 1$ - and so is a free $BP$-module. Thus, we can assume that $r$ is positive.

We have an action of Steenrod and Symmetric operations on $BP_{(r)}^*(X)$ - see
\cite[Pr.7.14]{SOpSt}, given by:
for $z\in\op{CH}^r(X)$, and $u\in BP_d$,
$$
St(z\cdot u)=z\cdot\iis^r\cdot t^{r(p-1)}\cdot St(u);\hspace{1cm}
\Phi(z\cdot u)=z\cdot\iis^r\cdot t^{r(p-1)}\cdot\Phi(u)_{\leq -r(p-1)},
$$
where we denote $\rho(z\otimes u)$ as $z\cdot u$, $\Phi(u)_{\leq -r(p-1)}$ is the
$t^{\leq -r(p-1)}$-part of $\Phi(u)$, and $\iis$ is the product of our
coset representatives - see (\ref{St}).

Let $\alpha=\sum_{j\in J}z_j\otimes u_j$ be an element of non-positive codimension $c=(r-d)$ such that
$\sum_{j\in J}z_j\cdot u_j=0\in BP_{(r)}^*(X)$. That is, $\alpha$ is a {\it relation}.
Let us say that $\alpha$ is {\it supported} in
$\{z_j,\,j\in J\}$.
We want to show that our element belongs to the $BP$-submodule $M_{>c}$ generated by similar
relations of larger codimension.

We will prove by the increasing induction on $m$ that this is true modulo $I(p)^m$.
The case $m=0$ is evident.
Suppose that all $u_j$'s belong to $I(p)^m$.
Denote as $\Phi(\alpha)$ the expression
$$
\sum_{j\in J}z_j\otimes\iis^r\cdot t^{r(p-1)}\cdot\Phi(u_j)_{\leq-r(p-1)}
\in\op{CH}^r(X)\otimes_{\zz}BP[t^{-1}].
$$
It is again a relation (becomes zero if we substitute $\otimes$ by $\cdot\,$), because
$\rho(\Phi(\alpha))=\Phi(\rho(\alpha))=0$.
Moreover, it belongs to $\op{CH}^r(X)\otimes I(p)^{m-1}[t^{-1}]$ by the Proposition \ref{SymIm}.
We have:
$$
(\square^p-St-[p]\cdot\Phi):BP^*\row BP^*[[t]]t.
$$
Since the $p$-th power $\square^p$ is zero on $BP_{(r)}^*$, for $r>0$, we obtain that
the $t$-non-positive components of
$St$ and $-[p]\cdot\Phi$ coincide. But by Corollary \ref{StId}, the $t^{(r-d)(p-1)}$-component of
$St(\alpha)$ is congruent to $\pm\alpha$ modulo $I(p)^{m+1}$ (notice that $\iis\equiv -1\,(mod\,p)$).
Since $(r-d)(p-1)\leq 0$, this component
can be written as a sum of various components of $-\Phi(\alpha)$ (which are all "relations")
multiplied by various $P_l$, $l\geq 0$.
Notice, that all our elements are still supported in $\{z_j,\,j\in J\}$.
Separating $P_0$ from $P_l,l>0$, we get:
$\alpha\equiv p\alpha_1+\beta_1\,\,\,(\,mod\,I(p)^{m+1})$, where
$\alpha_1$ and $\beta_1$ are relations supported in $\{z_j,\,j\in J\}$, where $\beta_1\in M_{>c}$
and the coefficients of $\beta_1$ are in $I(p)^m$, while the coefficients of $\alpha_1$ are in $I(p)^{m-1}$.
Repeating these arguments for $\alpha_1,\alpha_2,\ldots$, we get that
$\alpha\equiv\beta\,\,\,(\,mod\,I(p)^{m+1})$, where $\beta\in M_{>c}$ has coefficients in $I(p)^m$
(and the same support). This finishes the proof of our $m$-inductive statement.

Taking $m$ large enough, and using the fact that $\ddim(P_i)>0$, for $i>0$,
we obtain that for arbitrary codimension $c\leq 0$ relation
$\alpha$ with support $\{z_j;j\in J\}$, and for arbitrary $n$, there exists $\beta\in M_{>c}$ with
the same support, such that $(\alpha-\beta)=p^n\cdot\gamma$ for some relation $\gamma$ with
mentioned support.

Consider $N=\oplus_{j\in J}z_j\cdot BP_d$ - the free $\zz_{(p)}$-module of finite rank.
Let $K$ be the submodule of relations, and $L$ be the submodule consisting of elements from $M_{>c}$.
Then $K\supset L$, and as we saw above,  $K/L$ is infinitely divisible by $p$. Hence, $K=L$,
and so, any relation of non-positive codimension can be expressed through relations of larger codimension.
Theorem is proven.
\Qed
\end{proof}

\begin{rem}
Notice, that without the use of the non-additive $0$-th Symmetric operation $\Phi^{t^0}$ we would
be able only to prove that relations are concentrated in non-negative codimensions.
\end{rem}

\begin{thm}
\label{osn}
Let $X$ be a smooth quasi-projective variety. Then $\Omega^*(X)$ as a module over $\laz$
can be defined by generators of non-negative codimensions and relations of positive codimensions.
\end{thm}

\begin{proof}
The statement follows from the graded case by induction on the dimension of support, using the fact
that the "minimal" generators of $Gr\Omega^*$ can be chosen as projections of "minimal" generators
for $\Omega^*$ (generators of Chow groups), and
any relation there can be lifted to one in $\Omega^*$ (proven, again, by a similar
induction).
\Qed
\end{proof}

This extends the result of M.Levine and F.Morel - \cite[Th.1.2.19]{LM} claiming that the generators
of $\Omega^*(X)$ are in non-negative codimensions.

Applying the above Theorem to the case of a curve, we obtain:

\begin{thm}
\label{curve}
Let $C$ be a smooth curve over $k$. Then:
$$
\Omega^*(C)=\laz\cdot 1\oplus\op{CH}_0(C)\otimes_{\zz}\laz
$$
\end{thm}

\begin{proof}
Indeed, since $\Omega^*$ is {\it constant},
we have natural identifications $Gr\Omega^*(C)=\Omega^*(C)$ and $\Omega^*_{(0)}(C)=\laz$.
At the same time, all the relations in $\Omega^*_{(1)}(C)$ are in codimension $1$, and so are already
accounted for in $\Omega_0(C)=\op{CH}_0(C)$.
\Qed
\end{proof}

Notice, that this result can be obtained also with the help of the algebro-geometric version of
Atiyah-Hirzerbruch spectral sequence (introduced in \cite{HM} and \cite{Ho}) using arguments
similar to that of \cite{Lcomp}.
At the same time, the author does not see how to prove the Theorem \ref{osn} using this method.

On the base of the small knowledge we possess about $\Omega^*$,
I would make the following guess about {\it syzygies} of our module:

\begin{conj}
\label{syz}
Let $X$ be a smooth variety of dimension $d$.
Then $\Omega^*(X)$ as a module over $\laz$ has a free resolution whose $j$-th term has
generators concentrated in codimensions between $j$ and $d$.
In particular, the cohomological dimension of the $\laz$-module $\Omega^*(X)$ is $\leq d$.
\end{conj}

Note, that the cohomological dimension of the $\laz$-module $\Omega^*(X)$ together with the
information on where the generators of syzygies are located should be an interesting invariant
of a variety measuring it's complexity.

\begin{exa}
\label{Rost}
Rost motives:
\end{exa}
Aside from the case of a curve, one of the few cases where the ring $\Omega^*(X)$ was computed
is that of a Pfister quadric $Q_{\alpha}$, where $\alpha=\{a_1,\ldots,a_n\}$ is a {\it pure symbol}
in $K^M_n(k)/2$. The cobordism motive $M^{\Omega}(Q_{\alpha})$ splits as a direct sum
of (cobordism) {\it Rost-motives} - see \cite{VY}:
$$
M^{\Omega}(Q_{\alpha})=\oplus_{i=0}^{2^{n-1}-1}M^{\Omega}_{\alpha}(i)[2i].
$$
The Rost-motive $M^{\Omega}_{\alpha}$ can be realized as a motive of affine $(2^{n}-2)$-dimensional
quadric $A=Q_{\alpha}\backslash P$, where $P$ is any codimension $1$ subquadric of $Q_{\alpha}$.
Over algebraic closure $\kbar$ the Rost-motive splits into the sum of two Tate-motives:
$M^{\Omega}_{\alpha}|_{\kbar}=\llaz\oplus\llaz(2^{n-1}-1)[2^n-2]$.

It was shown in \cite[Th.3.5]{VY} that the restriction to the algebraic closure homomorphism
is injective on $\Omega^*(M^{\Omega}_{\alpha})$ and
$$
\Omega^*(M^{\Omega}_{\alpha})=\laz\cdot e^0\oplus I(2,n-2)\cdot e_0,
$$
where $e^0$ is the generic cycle while $e_0$ is represented by any middle-dimensional plane on
$Q_{\alpha}|_{\kbar}$,
and $I(2,n-2)\subset\laz$ is the ideal generated by the classes of the varieties of dimension
$\leq 2^{n-2}-1$ whose all characteristic numbers are even.

The ideal $I(2,n-2)$ is generated by $2=v_0,v_1,\ldots,v_{n-2}$, where $\ddim(v_i)=2^i-1$.
These $v_i$'s correspond to the classes $P_i$ in the $BP$-theory.
Since $v_i$ form a regular sequence in $\laz$, the {\it Koszul complex}
$R_{*}=\ov{\Lambda}_{\laz}^{*-1}(v_0,\ldots,v_{n-2})$ (where $\ov{\Lambda}^{*}$
is the positive degree part of the external algebra $\Lambda^{*}$) with the natural differential is
the free $\laz$-resolution of $I(2,n-2)$.

For $I\subset\{0,\ldots,n-2\}$, the codimension of the generator $(\Lambda_{i\in I}v_i)\cdot e_0$
of our resolution will be $(2^{n-1}-1)-\sum_{i\in I}(2^i-1)$. Notice, that this number belongs to
the range between $|I|$ and $(2^{n-1}-1)$. Moreover, for the highest generator of
the external algebra $(\Lambda_{i\in\{0,\ldots,n-2\}}v_i)\cdot e_0$ the codimension will be
exactly $(n-2)$. And the respective group
$\op{Tor}^{\laz}_{n-2}(\Omega^*(M^{\Omega}_{\alpha}),\zz/2)$ is really non-trivial.
This shows that the estimate in the Conjecture above can not be improved.

Since the (cobordism) motives of the Pfister quadrics, and more generally, {\it excellent quadrics}
(see \cite{Kn2}, or \cite[Introduction]{EC} for the definition)
are direct sums of Rost-motives (see \cite[Pr.2.4]{R} and \cite[Cor.2.8]{VY})
we easily get:

\begin{prop}
\label{exc}
The Conjecture \ref{syz} is true for excellent quadrics. In particular, for Pfister quadrics.
\end{prop}

The same can be seen for the {\it generalized Rost motive} for odd primes, at least $p$-adically,
from the computations of N.Yagita - see \cite{YaABPNV}.

This provides a tiny bit of evidence in support of the Conjecture.

\end{document}